\documentclass[11pt,reqno]{amsart}

\def\volume{\operatorname{vol}}

\def\op{\operatorname}

\usepackage[usenames,dvipsnames]{color}
\usepackage{amsmath,amssymb,bm,amsthm, amsfonts, mathrsfs, eucal, epsfig}
\usepackage{graphicx}
\usepackage{url}
\allowdisplaybreaks
\makeatletter 
\@addtoreset{equation}{section}
\makeatother  

\begin{document}

\newtheorem{Thm}{Theorem}[section]
\newtheorem{Def}[Thm]{Definition}
\newtheorem{Lem}[Thm]{Lemma}
\newtheorem{Rem}[Thm]{Remark}

\newtheorem{Problem}[Thm]{Problem}

\newtheorem{Cor}[Thm]{Corollary}
\newtheorem{Prop}[Thm]{Proposition}
\newtheorem{Example}[Thm]{Example}
\newcommand{\g}[0]{\textmd{g}}
\newcommand{\pr}[0]{\partial_r}
\newcommand{\dif}{\mathrm{d}}
\newcommand{\bg}{\bar{\gamma}}
\newcommand{\md}{\rm{md}}
\newcommand{\cn}{\rm{cn}}
\newcommand{\sn}{\rm{sn}}
\newcommand{\seg}{\mathrm{seg}}

\newcommand{\Ric}{\mbox{Ric}}
\newcommand{\Iso}{\mbox{Iso}}
\newcommand{\ra}{\rightarrow}
\newcommand{\Hess}{\mathrm{Hess}}
\newcommand{\RCD}{\mathsf{RCD}}

\title{2-regular points in Ricci limit spaces}
\author{Lina Chen}
\address[Lina Chen]{School of Mathematics and Statistics, Nanjing University of Science and Technology, Nanjing China}

\email{chenlina@njust.edu.cn}
\thanks{Supported partially by NSFC Grant 12471046, Jiangsu NSF Grant BK20240185 and the Fundamental Research Funds for the Central Universities No. 30923010201.}

\maketitle

\begin{abstract}

\setlength{\parindent}{10pt} \setlength{\parskip}{1.5ex plus 0.5ex
minus 0.2ex} 

In this note, we will  show that if a measured Gromov-Hausdorff limit space of a sequence of Riemannian manifolds with lower Ricci curvature bound contains a  $2$-regular point which lies in the interior of a geodesic, then it is $2$-rectifiable. And we will also give some properties about $2$-regular points.

 \end{abstract}

\section{Introduction}

A Ricci limit space $(X, d, \nu, x)$ is a measured Gromov-Hausdorff limit of a sequence of complete $n$-manifolds $(M_i, g_i, \underline{\volume}_i, x_i)$ with Ricci curvature lower bound $\op{Ric}_{M_i}\geq -(n-1)$ where the renormalized measure $\underline{\volume}_i(\cdot)=\frac{\volume(\cdot)}{\volume(B_1(x_i))}$ (\cite{CC1}). For any $x\in X$, any $r_j\to 0$, by Gromov's precompactness theorem,  passing to a subsequence, $(r_j^{-1}X, x)=(X, r_j^{-1}d, x)$ is Gromov-Hausdorff convergent to a length space $(T_x, x^*)$ which is called a tangent cone at $x$. And $x$ is $k$-regular if each tangent cone at $x$ is isometric to $(\Bbb R^k, 0)$, for some integer $k>0$. The set of $k$-regular points is denoted by $\mathcal R_k$. The regular set $\mathcal R=\cup_k \mathcal R_k$ and the singular set $\mathcal S=X\setminus \mathcal R$. By \cite{CC1}, $\mathcal R$ has full measure and in the
 non-collapsing case,  i.e., $\volume(B_1(x_i))\geq v>0$ for all $i$ large, $\mathcal R=\mathcal R_n$ and 
 $\nu$ is  equivalent to the $n$-dimensional Hausdorff measure $\op{Haus}^n$. In the collapsing case, i.e., $\volume(B_1(x_i))\to 0$, Colding-Naber \cite{CN} improved the result in \cite{CC1} by showing that there is a unique integer $k$ such that $\mathcal R_k$ has full measure and thus $X$ is $k$-rectifiable.   And by \cite{Hon} (cf. \cite{KL}), for any $l>k$, $\mathcal R_l=\emptyset$. A problem is 
\begin{Problem}
Assume that a sequence of $n$-manifolds $(M_i, g_i, \underline{\volume}_i, x_i)$ is measured Gromov-Hausdorff convergent to $(X, d, \nu, x)$ with $\op{Ric}_{M_i}\geq -(n-1)$ and $\volume(B_1(x_i))\to 0$. If $(X, d, \nu)$ is $k$-rectifiable, is there $l<k$ such that $\mathcal R_l\neq \emptyset$?
\end{Problem}

In  \cite{Ch}, we have showed that if $\mathcal R_1\neq \emptyset$, then the Ricci limit space $X$ is a $1$-dimensional topological manifold. And the same result holds for $\op{RCD}(K, N)$-spaces by \cite{KiLa}.

 In this note we will study the $2$-regular points and show that
 \begin{Thm} \label{main-thm}
 If a Ricci limit space $(X, d, \nu)$ contains a $2$-regular point which lies in the interior of a geodesic, then $\mathcal R=\mathcal R_2$, i.e., $X$ is $2$-rectifiable. 
 \end{Thm}
 
Note that for $k\geq 2$, $k$-regular points may lie in no geodesics like the pole of a paraboloid. By \cite{CN2}, for $k\geq 3$, there is a Ricci limit space which contains a $k$-regular point $y$ that for any two unit speed geodesics $\gamma_1, \gamma_2$ from $y$, any $\theta\in [0, \pi]$, there is $t_i\to 0$ such that the angle $\measuredangle\gamma_1(t_i)y\gamma_1(t_i)$ converges to $\theta$. Here we will show that for $2$-regular point, this will not happen.
\begin{Prop} \label{main-2}
For a Ricci limit space $(X, d, \nu)$ and $x\in \mathcal R_2$,  there are $\epsilon_0>0, r_0>0$ such that for $0\leq \epsilon\leq \epsilon_0, 0<r\leq r_0$ if 
$d_{GH}(B_r(x), B_r(0))\leq \epsilon r$ and 
$x_1, x_2\in X$ satisfying that $d(x_1, re_1)\leq r\epsilon, d(x_2, -re_1)\leq r\epsilon$, for $c_i$ which is a unit speed minimal geodesic from $x$ to $x_i$, $i=1,2$, we have that for any $s$, $$\frac{d(c_1(s), c_2(s))}{2s}\geq 1-\Psi(\epsilon),$$
where $\Psi(\epsilon)\to 0$ as $\epsilon\to 0$, $B_r(0)\subset \Bbb R^2$ and $e_1$ is a unit vector in $\Bbb R^2$.
\end{Prop} 

 Another open problem in collapsing Ricci limit spaces is the following.
\begin{Problem}\cite[Open problem 3.4]{Na}
Assume that a sequence of $n$-manifolds $(M_i, g_i, \underline{\volume}_i, x_i)$ is measured Gromov-Hausdorff convergent to $(X, d, \nu, x)$ with $\op{Ric}_{M_i}\geq -(n-1)$ and $\volume(B_1(x_i))\to 0$. If $(X, d, \nu)$ is $k$-rectifiable, is there an open subset of full $\nu$-measure in $X$ that is a $k$-dimensional topological manifold ? 
\end{Problem}

In the non-collapsing case, \cite{CC1} shows that a Ricci limit space has an open full measure subset which admits a smooth manifold structure. In the collapsing case, if $k=1$, by \cite{Hon}  (cf. \cite{Ch}), $X$ is a $1$-dimensional topological manifold and for $k\geq 3$ there are counterexamples \cite{HNW, Zhou} of Problem 1.4 ($k\geq 4$ were constructed in \cite{HNW} and then \cite{Zhou} gave counterexamples for $k\geq 3$). For $k=2$, we know that almost all $2$-regular points lie in the interior of some geodesics. And the following structure holds:

 \begin{Thm} \label{pent-1}
 If $x\in \mathcal R_2$ lies in the interior of a geodesic $\gamma$, then there is $\epsilon>0$ such that there is a ``$(2,\epsilon)$-strainer" at $x$.
 \end{Thm}
 
 In Theorem\ref{pent-1}, the ``$(2,\epsilon)$-strainer" is a little different with the one in Alexandrov spaces where the angle of two geodesics beginning with the same point can be well defined however is impossible for metric measure spaces with some kind of  Ricci curvature lower bound (see the discussion before the proof of Theorem\ref{pent-1} in Sec. 3). We call it ``strainer" for the same property as in Alexandrov spaces that there are geodesics such that in any small scale, they are close to a norm frame in a Euclidean space (see the proof of Theorem\ref{pent-1}). 
 In Alexandrov spaces, if there is a $(m, \epsilon)$-strainer at $x$ where $m$ equals the local strainer number at $p$ (the supremum $m$ such that there exists an $m$-strainer), then it is also a strainer for each points in a neighborhood $U_x$ of $x$ and thus $U_x$ is bi-Lipschitz homeomorphic to an open region in $\Bbb R^m$ (cf. \cite{BBI}). However for Ricci limit spaces (or $\op{RCD}$-spaces), so far, this strainer's properties is not sufficient.

 \begin{Rem}
 \begin{enumerate}
\item[(1)] For a $2$-dimensional non-collapsing Ricci limit space, by \cite{LS}, it is a Alexandrov space. In fact, \cite{LS} derived this result in more general non-collapsing metric measure spaces (called non-collapsed $\op{RCD}(K, 2)$-spaces).
\item[(2)]For a $k$-regular point, $k\geq 3$, that lies in the interior of a geodesic, we also have some properties like Theorem\ref{pent-1} (see Proposition 3.2).
\item[(3)] For $\op{RCD}(K, N)$-spaces, Theorem\ref{main-thm} holds if Kapovitch-Li's result \cite{KL} (Theorem 2.6) holds. 
\end{enumerate}
\end{Rem}

The author would like to thank Bangxian Han, Nan Li, Xueping Li and Jiayin Pan's many helpful discussions and advices in the first version.

\section{Preliminaries}

In this section, we will supply some definitions and properties that will be used in the proofs of the main results.
 
 Assume a sequence of complete $n$-manifolds $(M_i, g_i, \underline{\volume}_i, x_i)$ is measured Gromov-Hausdorff convergent to $(X, d, \nu, x)$ with $\op{Ric}_{M_i}\geq -(n-1)$ and $\underline{\volume}_i=\frac{\volume}{\volume(B_1(x_i))}$ where we denote it by $(M_i, g_i, \underline{\volume}_i, x_i)\to (X,  d, \nu, x)$ or $d_{mGH}((M_i, , g_i, \underline{\volume}_i, x_i), (X, d, \nu, x))\to 0$. If we only consider the Gromov-Hausdorff convergence, we also write as $(M_i, x_i)\to (X, x)$ or $d_{GH}((M_i, x_i), (X, x))\to 0$. $(X, d, \nu)$ is called a Ricci limit space. We will use the following splitting theorem (\cite{CC}) and sharp H\"older continuity of tangent cones (\cite{CN}) in Ricci limit spaces from time to time.
  \begin{Thm}[Splitting theorem \cite{CC}]
 Let $(X, x)$ be a Gromov-Hausdorff limit of a sequence of manifolds $(M_i, x_i)$ with $\mathrm{Ric}_{M_i}\geq -(n-1)\delta_i, \delta_i\rightarrow 0.$ If $X$ contains a line, then $X$ splits isometrically with a $\mathbb R^1$-factor i.e., $X=\mathbb{R}^1\times Y$ for some length space $Y$.
 \end{Thm}
 
By the splitting theorem, for any tangent cone in a Ricci limit space $(X, d, \nu)$, if it contains a line, it will split.  A geodesic $\gamma:[0, l]\to X$ is called a limit geodesic if there are geodesics $\gamma_i:[0, l_i]\to M_i$ with $l_i\to l$ and $\gamma_i\to \gamma$ pointwise.   In a metric space $(X, d)$, two unit speed geodesics $\gamma_1$ and $\gamma_2$ that defined on the unit interval  are said branching if there is $t\in (0, 1)$ such that $\gamma_1(s)=\gamma_2(s)$ for all $s\in [0, t]$ and $\gamma_1(s)\neq \gamma_{2}(s)$ for all $s\in (t, 1]$.  It was proved in \cite{Deng} that geodesics cannot branch in a Ricci limit space (in fact, Deng derived this result in a more general case).  
 
 \begin{Thm}[\cite{Deng}]
 Any geodesic in a Ricci limit space is non-branching.
 \end{Thm}
 By above theorem, geodesics in a Ricci limit space are all limit geodesics where we will still write as limit geodesics sometimes.  
 
  \begin{Thm}[Sharp H\"older continuity \cite{CN}]\label{holder}
 Let $(X, d,\nu)$ be a Ricci limit space. And let $\gamma:[0, l]\to X$ be a unit speed limit geodesic. There exist $\alpha(n), C(n), r_0(n)>0$ such that for any $\delta>0$, $0<r<r_0\delta l$, $s, t\in(\delta l, l-\delta l)$,
 $$d_{GH}(B_r(\gamma(s)), B_r(\gamma(t)))\leq \frac{C}{\delta l}r|s-t|^{\alpha(n)}.$$
 \end{Thm}
 The sharp H\"older continuity implies that 
  \begin{Thm}[\cite{CN}] \label{ae-full}
 Let $(X, d, \nu)$ be a Ricci limit space. Then
 \begin{enumerate}
\item[(2.4.1)]
For $\nu \times \nu$ almost every pair $(a_1, a_2)\in A_1 \times A_2$, where $A_1$
and $A_2$ are subsets of $X$ that contained in a bounded ball, there exists a limit
geodesic from $a_1$ to $a_2$ whose interior lies in some $\mathcal{R}_l$,
$l$ is an integer;
\item[(2.4.2)]
For $\nu \times \nu$ almost every pair $(x, y)\in X \times X$, they are in the interior of a limit geodesic; 
 \item[(2.4.3)]
  There is an integer $k$, such that $\nu(X\setminus \mathcal{R}_k)=0$.
  \end{enumerate}
\end{Thm} 

Theorem \ref{ae-full} implies that in a Ricci limit space $(X, d, \nu)$, there is $k$ such that the subset of points which are $k$-regular and lie in the interior of a limit geodesic has full $\nu$-measure.
Compared with the sharp H\"older continuity, in \cite{KL}, Kapovitch and Li derived that 
 \begin{Thm}[\cite{KL}] \label{same-dim}
 For a constant speed geodesic $\gamma: [0, 1]\to X$ where $X$ is a Ricci limit space, the rectifiable dimension of same scale tangent cone at $\gamma(t)$ is constant for any $t\in (0,1)$.
 \end{Thm}
 
 This theorem and the result in \cite{Hon} or \cite{Ch} implies that
 \begin{Lem}
 For a Ricci limit space $X$, if one interior point of a geodesic $\gamma $ in $X$ lies in $\mathcal R_2$, then the whole interior of $\gamma$ is in $\mathcal R_2$.
 \end{Lem}
 
 In fact by Theorem\ref{same-dim}, for any interior points $\gamma(t)$ of $\gamma$, the tangent cone $T$ at $\gamma(t)$ has rectifiable dimension $2$. And thus \cite{Hon} or \cite{Ch} implies that $T$ is a half plane or $\Bbb R^2$. And by the sharp H\"older continuity in Ricci limit spaces \cite{CN} (Theorem 2.2), $T$ is $\Bbb R^2$. 
 
 \begin{Def} A metric measure space $(X, d, \nu)$ is called $k$-rectifiable if there are countable $X_i\subset X,  i\in \mathbb N$ such that for each $i$,
 \begin{enumerate}
\item[(2.4.1)] $X_i$ is $\op{Haus}^k$-measurable;
 \item[(2.4.2)] There exists bi-Lipschitz map $\phi_i: X_i\to \mathbb R^k$;
 \item[(2.4.3)] $\nu(X\setminus \cup_{i}X_i)=0$.
 \end{enumerate}
 \end{Def}
 By \cite{CC3} and \cite{CN}, any Ricci limit space is $k$-rectifiable for some integer $k$
 \begin{Thm}[\cite{CC3, CN}]
 Let $(X, d, \nu)$ be a Ricci limit space of a sequence of $n$-dimensional manifolds with lower Ricci curvature bound. Then there is a unique $0<k \leq n$ such that $(X, d, \nu)$ is $k$-rectifiable where $k$ is the dimension of the full $\nu$-measure $k$-regular set $\mathcal R_k$ of $X$. 
 \end{Thm}

\section{Proofs of the main results}

To derive Theorem\ref{main-thm},  we only need to show that if the whole interior of a geodesic $\gamma$ lies in $\mathcal R_2$, the minimal geodesics between two points which are close with $\gamma$ and lie in the different ``sides" of $\gamma$ intersect with $\gamma$ (see Proposition\ref{main}) and then a similar argument as in \cite{Ch} will give the result. And Proposition 1.3 and Theorem\ref{pent-1} are all derived by a similar observation as in the proof of  Proposition\ref{main}.


 Let $(X, d, \nu)$ be a Ricci limit space. For $x\in \mathcal R_2$, for each $\epsilon_0>0$, there is $r_0>0$ such that for any $r\leq r_0$
 $$d_{GH}(B_r(x), B_r(0))\leq \epsilon_0 r, \,B_r(0)\subset \mathbb R^2.$$
Let $e_1, e_2$ be an orthonormal basis in $\mathbb R^2$. Then for any $\theta\in [0, \pi)$, there are $x_r^{+\theta}, \,x_r^{-\theta}\in X$ such that 
$$d(x_r^{+\theta}, \cos\theta r e_1+\sin\theta re_2)\leq \epsilon_0 r, \quad d(x_r^{-\theta}, -\cos\theta r e_1-\sin\theta re_2)\leq \epsilon_0 r.$$
By the choice of $x_r^{+\theta}, \,x_r^{-\theta}\in X$,
$$\left|d(x, x_r^{\pm\theta})-r\right|\leq \epsilon_0 r, $$
$$d(x, x_r^{-\theta})+d(x, x_r^{+\theta})-d(x_r^{+\theta}, x_r^{-\theta})\leq 3\epsilon_0 r.$$
For a unit speed minimal geodesic $\gamma$, if $x=\gamma(t), t\in (0, 1)$ and $x\in \mathcal R_2$, for $\epsilon>0, r>0$,  $\theta\in [0, \pi)$, we can take $x_r^{\pm\theta}\in X$ as above. Especially, in the following we will always take $x_r^{\pm 0}=\gamma(t\pm r)$.

\begin{Prop} \label{main}
Assume that a Ricci limit space $(X, d, \nu)$ contains a unit speed geodesic $\gamma: [0,1]\to X$ whose interior $\gamma^o$ lies in $\mathcal R_2$. Given $\pi>\theta_0>0$, $\delta>0$, there are $\epsilon_0>0$ and $\delta/2>r_0>0$ such that for any point $x\in \gamma_{\delta}=\left.\gamma\right|_{[\delta,1- \delta]}$,  any $0<r<r_0$, $\theta\in [\theta_0, \pi-\theta_0]$ and $x_r^{\pm\theta}$ as above,  the minimal geodesic $c$ between $x_r^{+\theta}$ and $x_r^{-\theta}$ intersect with $\gamma$ at a point near $x$.
\end{Prop}
\begin{proof}
We first show that for each $x\in \gamma_{\delta}$, we can find $\epsilon_0>0, r_0>0$ which may depend on $x$ such that the above statement holds.

Argue by contradiction. Assume that there are $\epsilon_i\to 0$, $r_i\to 0$, $x_{r_i}^{\pm \theta}$ such that the minimal geodesic $c_i$ between $x_{r_i}^{\pm\theta}$ does not intersect with $\gamma$.  By the choice of  $x_{r_i}^{\pm\theta}$,  we have that for $i$ large, the distance $d(x, c_i)\ll r_i$ and the point $z_i\in \gamma^o$ that $d(c_i, \gamma)=d(c_i, z_i)$ satisfies $d(x, z_i)\ll r_i$.  In fact, if $d(x, c_i)>C r_i$, for some constant $C>0$, when $r_i\to 0$, $(r_i^{-1}X, x)\to (\mathbb R^2, 0)$, $c_i$ goes to a limit geodesic $c_{\infty}$ that passing through $0$ which is a contradiction for that $d(0, c_{\infty})>C$.  Now let $d(x, c_i)= o(r_i)$ and assume $d(x, z_i)>C r_i$ for some $C>0$. Since $d(x, z_i)<3r_i$, we may assume $(r_i^{-1}X, x)\to (\mathbb R^2, 0), \gamma\to \gamma_{\infty}, z_i\to z_{\infty}$.  Then $d(0, c_{\infty})=0$, $d(z_{\infty}, c_{\infty})=0$, $3>d(0, z_{\infty})>C$, which is contradict to that $\gamma_{\infty}$ is the only line between $0$ and $z_{\infty}$.

Let $d_i=d(c_i, z_i)$. Then $d_i\leq d(c_i, x)\ll r_i$. Assume $d_i>0$ for any large $i$ and assume 
$$(d_{i}^{-1}X, z_{i})\to (Z, z).$$
By the sharp H\"older continuity of \cite{CN} (see Theorem \ref{holder}), for any $z_i\in \gamma^o$, for any $r$ small,
$$d_{GH}(B_r(z_i), B_r(x))\leq crd(x, z_i)^{\alpha}.$$
And thus $(d_{i}^{-1}X, x)\to (\mathbb R^2, 0)$
implies that $Z=\mathbb R^2$.

Note that there are two lines in $Z$, one is the limit $\gamma_{\infty}$ of $\gamma$ and the other is the limit $c_{\infty}$ of $c_{i}$. By the definition of $d_i$, $d(c_{\infty}, \gamma_{\infty})=1>0$ and thus $c_{\infty}$ and $\gamma_{\infty}$ are parallel. 

If $z_i=\gamma(t_i)$, let
$$l_{i1}=\max\{l, \,d(\gamma(t_i-\tau), c_i)< 2d_i, \forall\, \tau\in [0, l)\}.$$
Then by above discussion, 
$$\frac{l_{i1}}{d_i}=\frac{1}{\epsilon_{i1}}\to \infty, , \frac{l_{i1}}{r_i}\to 0, \text{ as }i\to \infty.$$

Let
$$l_{i2}=\max\{l, \,d(\gamma(t_i-l_{i1}-\tau), c_i)< 2^2d_i, \forall\, \tau\in [0, l)\}.$$
Then as above, we have 
$$\left((2d_i)^{-1}X, \gamma(t_i-l_{i1})\right)\to (\Bbb R^2, 0).$$
Assume that $c_i\to c_{\infty}$ and $\gamma\to \gamma_{\infty}$ and $\gamma_{\infty}(0)=0$. Then 
$$d(\gamma_{\infty}(\tau), c_{\infty})\leq 1, \forall\, \tau\in \left[0, \infty\right)$$
which implies that
$$\frac{l_{i2}}{2d_i}=\frac{1}{\epsilon_{i2}}\to \infty, i\to \infty,$$
for $c_{\infty}$ and $\gamma_{\infty}$ are two lines in $\Bbb R^2$.


Now by induction, for $k$, we have $l_{ik}$, 
$$l_{ik}=\max\{l, \,d(\gamma(t_i-\sum_{p=1}^{k-1}l_{ip}-\tau), c_i)< 2^kd_i, \forall\, \tau\in [0, l)\},$$
$$\frac{l_{ik}}{2^{k-1}d_i}=\frac{1}{\epsilon_{ik}}\to \infty, i\to \infty, $$

And for $k$ large,
$$\sum_{p=1}^k \frac{l_{ip}}{r_i}\leq 1, \quad 2^k d_i\geq \frac{r^i}{2}$$
which could not happen for $i$ large.

Secondly, by sharp H\"older continuity Theorem 2.3, we have that given $\delta>0$ small, for each point $x=\gamma(t), t\in [\delta, 1-\delta]$, $\pi>\theta_0>0$, there are $\epsilon_0>0, r_0>0$, such that for  $0<r<r_0$, $\theta\in [\theta_0, \pi-\theta_0]$,  if $x_i=\gamma(t_i)$, $r>|t_i-t|\to 0$, then for large $i$ and any $x_{i, r}^{\pm\theta}$,  the minimal geodesic $c_i$ between $x^{+\theta}_{i, r}$ and $x^{-\theta}_{i, r}$ intersect with $\gamma$ at some point.



Then the result follows by the closeness of $\gamma_{\delta}$.
\end{proof}
Now we derive Theorem\ref{main-thm} by Lemma 2.6 and Theorem\ref{ae-full} as in \cite{Ch}.
\begin{proof}[Proof of Theorem\ref{main-thm}]
By Lemma 2.6 we know that the interior $\gamma^0$ lies in $\mathcal R_2$. By Proposition\ref{main}, there are positive measured set $A_1, A_2$ such that minimal geodesics between $A_1$ and $A_2$ pass through $\gamma^0$ and thus lie in $\mathcal R_2$.
By (2.4.2), for almost all points in $A_1$ and $A_2$ are in the interior of some geodesics and thus in $\mathcal R_2$ which implies $\mathcal R=\mathcal R_2$ by (2.4.3) and \cite{Ch} (\cite{KiLa} for $\op{RCD}$-spaces).
\end{proof}

 For $k$-regular points, $k\geq 2$, by a similar argument as in Proposition\ref{main}, we have the following property. 
\begin{Prop} \label{trans}
Let $(X, d, \nu)$ be a Ricci limit space. Let $x\in \mathcal R_k\subset X$ lie in the interior of a unit speed minimal geodesic $\gamma$, i.e., $x=\gamma(t_0)\in \gamma_{\delta}$, for some $\delta>0$.  Then there are $\epsilon_0>0, \delta>r_0>0$ such that for $0\leq \epsilon<\epsilon_0, 0<r<r_0$, $x_r$ satisfying 
$$d(x, x_r)=r, \quad d(x_r, re_j)\leq \epsilon r, j\neq 1,$$
where $d_{GH}(B_r(x), B_r(0))\leq \epsilon r$, $e_1, \cdots, e_k$ are orthonormal basis of $\Bbb R^k$, and $d(\gamma(t_0+r), r e_1)\leq \epsilon r$,
 if $c(s)$ is a unit speed geodesic from $x$ to $x_r$, then for any $s$,
$$1-\frac{d(c(s),\gamma)}{s}\leq \psi(\epsilon).$$
 \end{Prop}
 \begin{proof}
 Argue by contradiction. Assume that there are $\epsilon_i\to 0, r_i\to 0$, $x_{r_i}$ and a unit speed geodesic $c_i$ from $x$ to $x_{r_i}$ such that there are $C<1$ and $s_i$ satisfying that
 $$\frac{d(c_i(s_i), \gamma)}{s_i} < C.$$
 
 Let 
 $s_i\ll r_i$ be the minimal one that  
 $$\frac{d(c_i(s_i), \gamma)}{s_i}\geq C, \, \forall\, s\geq s_i.$$
 
 As in the proof of Proposition 3.1, if $d(c(s),\gamma)=d(c(s), z_s)$ where $z_s\in \gamma$, then $d(x, z_s)\ll r_i$ and by sharp H\"older continuity,  $(d_i^{-1}X, z_{s_i})\to (\Bbb R^k, 0)$ where $d_i=Cs_i$.  Assume $c_i\to c_{\infty}$, $\gamma\to \gamma_{\infty}$. Then $c_{\infty}$ and $\gamma_{\infty}$ are two intersect lines in $\Bbb R^k$. 
By the choice of $s_i$, there are
 $s_{i1}=\max\{l, d(c(s_i+l), \gamma)<2d_i\}$, and
 $$\lim_{i\to\infty}\frac{s_{i1}}{d_i}=\frac1{C}.$$
 Then
 $$\frac{d(c_i(s_i+s_{i1}), \gamma)}{s_i+s_{i1}}<2d_i\leq Cs_i+(1+\delta_i)Cs_{i1}$$
 which is a contradiction to the choice of $s_i$ for $i$ large.  We could also use the following discussion.
 
 By induction, if $s_{ik}=\max\{l, d(c(s_i+\sum_{p=1}^{k-1}s_{ip}+l), \gamma)<2^kd_i\}$, then 
 $$\lim_{i\to\infty}\frac{s_{ik}}{2^{k-1}d_i}>\frac1{C}.$$ 
 
 However, for $k$ large, where $s_{i0}=s_i$,
 $$\sum_{p=0}^ks_{ip}= r_i-o(r_i), \quad 2^k d_i= r_i-o(r_i)$$ 
 hold for $i$ large, which can not happen. 
 \end{proof}
 
 In above proposition, for $k=2$, it seems that there is a $(2, \epsilon)$-strainer at $x$. Recall that in an Alexandrov space $X\in \op{Alex}(K)$, a $(k, \epsilon)$-stainer at $x$ is $k$ pairs of points $(a_i, b_i)$, such that
 $$\tilde \measuredangle a_ixb_i>\pi-\epsilon, \quad \tilde \measuredangle a_ixa_j>\pi/2-10\epsilon,$$
  $$\tilde \measuredangle a_ixb_j>\pi/2-10\epsilon, \quad \tilde \measuredangle b_ixb_j>\pi/2-10\epsilon,$$ 
  where $\tilde\measuredangle abc$ is the angle at $\bar b$ of comparison triangle $\bar a\bar b \bar c$ in the $K$-plane. 
  In a Ricci limit space $X$, for $2>\frac{d(b, a)}{d(b, c)}>\frac12>0$,  if we defined $\tilde\measuredangle \gamma_1 b\gamma_2=\liminf_{s\to 0} \tilde\measuredangle \gamma_1(s)b\gamma_2(s)$ where $\gamma_1$ and $\gamma_2$ are unit speed minimal geodesics from $b$ to $a$ and $c$ and the comparison triangle is in $\Bbb R^2$, we can also define ``$(k, \epsilon)$-strainer" at $b$ as in Alexandrov spaces for $k$ pairs of geodesics $(\gamma_{i1},\gamma_{i2})$. Note that in Alexandrov space, for a $(k, \epsilon)$-strainer $\{(a_i, b_i)\}_{i=1}^k$ at $x$, one may have $\frac{d(x, a_i)}{d(x, a_j)}, i\neq j$, very small where we have no definitions for this case in Ricci limit spaces. 
  
 \begin{proof}[Proof of Theorem\ref{pent-1}]
 Take $\epsilon_0$ and $r_0$ be the minimal one of Proposition\ref{main} and Proposition\ref{trans}. And for $r<r_0$, take $x_1=x_r^{\frac{\pi}{2}}$, $x_2=x_r^{-\frac{\pi}{2}}$. Let $c_i(s)$ be a unit speed minimal geodesic from $x$ to $x_i$, $i=1,2$. 
 By Proposition\ref{trans}, $c_i$ is almost perpendicular to $\gamma$. 
 
 Claim:  If  $d(c_1(s), x_s^{\frac{\pi}{2}})<\Psi(\epsilon) s$, then $d(c_2(s), x_s^{-\frac{\pi}{2}})<\Psi(\epsilon) s$. 
 
 In fact for any $s$ that $d(c_1(s), x_s^{\pm\frac{\pi}{2}})<\Psi(\epsilon) s$ and $d(c_2(s), x_s^{\mp\frac{\pi}{2}})<\Psi(\epsilon)s$, the minimal geodesic $c$ between $c_1(s)$ and $c_2(s)$ intersects with $\gamma$ and thus 
 $$\frac{d(c_1(s), c_2(s))}{s}\geq \frac{d(c_1(s), \gamma)+d(c_2(s), \gamma)}{s}\sim 2.$$ 
 And for $d(c_i(s), x_s^{\frac{\pi}{2}})<\Psi(\epsilon) s$ or $d(c_i(s), x_s^{-\frac{\pi}{2}})<\Psi(\epsilon)s$, $i=1,2$,
  $$\frac{d(c_1(s), c_2(s))}{s}\leq  \frac{2\Psi(\epsilon)s}{s}= 2\Psi(\epsilon)\ll 1.$$ 
 The claim derived  by the continuity of the function $\frac{d(c_1(s), c_2(s))}{s}$ and the initial choice of $x_1, x_2$.
 
  Then $\gamma$, $c_1$ and $c_2$ give a ``$(2, \Psi(\epsilon))$-strainer" at $x$.
 \end{proof}
 
For $2$-regular points which do not lie in any geodesics, we can also control the ``angle" of geodesics those are almost in the same geodesic. Proposition\ref{main-2} also shows that $2$-regular points are different with $k$-regular points for $k\geq 3$ by comparing with the example in \cite{CN2} (see the introduction).
\begin{proof}[Proof of Proposition\ref{main-2}]

Claim: There are $\epsilon_0>0, r_0>0$ such that for $0\leq \epsilon\leq \epsilon_0, 0<r\leq r_0$, $x_1, x_2\in X, c_1, c_2$ as the assumption and for $\gamma$, a minimal geodesic between $x_1$ and $x_2$, $d=d(x, \gamma)$, 
\begin{equation}\frac{d(c_1(d), c_2(d))}{2d}\geq 1-\Psi(\epsilon).\label{para-1}\end{equation}

For $\epsilon_i\to 0$, $r_i\to 0$, $d_{GH}(B_{r_i}(x), B_{r_i}(0))\leq \epsilon_i r_i$, take $x^i_1, x^i_2\in X$ satisfying that $d(x^i_1, r_ie_1)\leq r_i\epsilon_i$, $d(x^i_2, -r_ie_1)\leq r_i\epsilon_i$. Let $c^i_j$ be a unit speed minimal geodesic from $x$ to $x^i_j, j=1,2$, let $\gamma_i$ be a unit speed minimal geodesic between $x^i_1$ and $x^i_2$ and let $d_i=d(x, \gamma_i)=d(x, z_i)$. Without loss generality, assume $d_i>0$. Then by assumption we have that $d_i\ll r_i$ and $z_i$ is an almost middle point of $\gamma_i$, i.e., $|d(x^i_1, z_i)-d(x^i_2, z_i)|\leq \Psi(\epsilon_i)r_i$. 
Assume that $(d_i^{-1}X, x)\to (\Bbb R^2, 0)$, $\gamma_i\to \gamma$, $c^i_j\to c_j$, $j=1,2$. By assumption and the non-branching Theorem 2.2, the ray $c_j$ does not intersect with the line $\gamma$. We claim that $c_j$ is parallel with $\gamma$, $j=1,2$.

If $c_2$ is not parallel with $\gamma$, then the line that $c_2$ lies in which we still denote it by $c_2$ intersects with $\gamma$ at $y_2$. 
Then there is $c>0$ such that $\frac{d(c_2(s), \gamma)}{s}> c$, for all $s>0$. Let
$$l_i=\max\left\{l, \frac{d(c_2^i(s), \gamma_i)}{s}\geq c, \forall\, s\in (0, l) \right\}.$$ 
Then $l_i<r_i$ and
$$\frac{l_{i}}{d_i}=\frac1{\epsilon'_{i}}\to \infty, i\to \infty.$$

Let $d_{i2}=d(c^i_2(l_i), \gamma_i)=d(c^i_2(l_i), z_i^2)$.
By sharp H\"older continuity, 
$$d_{GH}(B_r(z_i^2), B_r(z_i))\leq cr d^{\alpha}(z_i^2, z_i).$$
And 
$$d_{GH}(B_{d_{i2}}(z_i), B_{d_{i2}}(x))\leq d(z_i, x)=d_i, \quad d_{i2}^{-1}d_i\leq (cl_i)^{-1}d_i\to 0,$$
thus
$$(d_{i2}^{-1} X, z_i^2)\to (\Bbb R^2, 0).$$
Assume that $c^i_2\to c_2, \gamma_i\to \gamma$, are two lines in $\Bbb R^2$. Then
$$\frac{d(c_2(s), \gamma)}{s}> c$$
still holds. And thus for $i$ large, there is $l>2\frac{l_{i}}{d_{i}}$ such that
$$\frac{d(c_2^i(s), \gamma_i)}{s}\geq c, \forall\, s\in (0, l),$$
a contradiction to the choice of $l_{i}$.

Next,  by the continuity of 
$\frac{d(c^i_1(s), c^i_2(s))}{2s}$, we have that 
$$\frac{d(c^i_1(s), c^i_2(s))}{2s}>1-\Psi(\epsilon_i).$$
In fact, if $c_1, c_2$ are the same ray, then
$$\frac{d(c^i_1(s), c^i_2(s))}{2s}\leq \Psi(\epsilon_i)$$ for $i$ large and $s$ small which is contradict to 
$$\frac{d(c^i_1(r_i), c^i_2(r_i))}{2r_i}\geq 1-2\epsilon_i.$$

Now we derive the result from the above claim. Argue by contradiction. Assume that there are $\epsilon_i\to 0$, $r_i\to 0$, $x^i_1, x^i_2\in X, c^i_1, c^i_2, \gamma_i$ as the above and $s_i$ such that 
\begin{equation}\frac{d(c^i_1(s_i), c^i_2(s_i))}{2s_i}<C<1.\label{para-2}\end{equation}
If $d_i\ll s_i$, then 
$$(s_i^{-1}X, x)\to (\Bbb R^2, 0)$$
and \eqref{para-1} implies that the limits of $c^i_1$ and $c^i_2$ lie in the same geodesic which is also the limit of $\gamma_i$. And thus 
\begin{equation}\frac{d(c_1^i(s_i), c_2^i(s_i))}{2s_i}\to 1, \label{para-3}\end{equation}
a contradiction to \eqref{para-2}.

If $s_i/d_i\to c$, then by the discussion of claim, we have that for $(s_i^{-1}X, x)\to (\Bbb R^2, 0)$ the limits of $c^i_1$ and $c^i_2$ lie in the same geodesic which are parallel with the limit of $\gamma_i$ and thus \eqref{para-3} holds.

If $s_i\ll d_i$, comparing $s_i$ with the distance of $x$ and the minimal geodesic between $c^i_1(d_i)$ and $c^i_2(d_i)$ and then using the induction argument to derive the result. 
\end{proof}

In Proposition\ref{trans}, we discuss geodesics that are almost perpendicular to a geodesic $\gamma$ in the interior at a regular point. For $2$-regular point, by Proposition\ref{main-2}, for the geodesic that are almost parallel with $\gamma$, we have that
\begin{Cor} \label{para-6}
Assume that $x\in \mathcal R_2$ lies in the interior of a geodesic $\gamma$, i.e., $x=\gamma(t_0)$. There are $\epsilon_0>0, r_0>0$ such that for $0<\epsilon\leq \epsilon_0, 0<r\leq r_0$ if $d_{GH}(B_r(x), B_r(0))\leq \epsilon r$, $x_1\in X$ satisfying that $d(x_1, re_1)\leq r\epsilon$ and $d(\gamma(t_0+r), re_1)\leq r\epsilon$, then for $c_1$, a unit speed minimal geodesic from $x$ to $x_1$ and for any $s$, $$\frac{d(c_1(s), \gamma)}{s}\leq \Psi(\epsilon).$$
\end{Cor}

\begin{Rem} \label{continuity}
As the discussion in Proposition\ref{main}, by sharp H\"older continuity Theorem\ref{holder}, we can see that if $\gamma_{\delta}\subset \mathcal R_2$, there are $\epsilon_0(\delta)>0$ and $r_0(\delta)>0$ such that Proposition\ref{trans} and Corollary\ref{para-6} holds for any $\epsilon<\epsilon(\delta), r<r_0(\delta), x\in \gamma_{\delta}$.
\end{Rem}


\begin{thebibliography}{10}

\bibitem{BBI} D. Burago, Y. Burago, S. Ivanov, A course in Metric geometry, American Mathematical Society

\bibitem{CC} J. Cheeger; T.H. Colding, Almost rigidity of warped products and the structure of spaces with Ricci curvature bounded below, Ann. of Math. (2) 144 (1996), 189-237. MR 1405949.

\bibitem{CC1} J. Cheeger; T.H. Colding, On the structure of spaces with Ricci curvature bounded below I, J. Diff. Geom., 46 (1997), $406-480$. MR 148488


\bibitem{CC3} J. Cheeger; T.H. Colding, On the structure of spaces with Ricci curvature bounded below III, J. Diff. Geom., 54 (2000), $37-74$. MR 1815411

\bibitem{Ch} L. Chen, A remark on regular points of Ricci limit spaces, Front. Math. China, 11(1), (2016), 21-26

\bibitem{CN} T. Colding; A. Naber, Sharp H\"older continuity of tangent cones for spaces with a lower Ricci curvature bound and applications, Ann. of Math., 176 (2012), 1172-1229

\bibitem{CN2} T. Colding, A. Naber, Lower Ricci curvature, branching and the bilipschitz structure of uniform Reifenberg spaces, Adv. Math., Volume 249, 2013, 348-358

\bibitem{Deng} Q. Deng, H\"older continuity of tangent cones in $\op{RCD}(K, N)$ spaces and applications to non-branching, arXiv: 2009.07956


\bibitem{Hon}\label{Honda} S. Honda. On low-dimensional Ricci limit spaces. Nagoya Math. J, 2013, 209: 1-22.

\bibitem{HNW} E. Hupp, A. Naber, K.-H. Wang: Lower Ricci Curvature and Nonexistence of Manifold Structure,preprint, arxiv: 2308.03909.

\bibitem{KL} V. Kapovitch and N. Li.  On dimensions of tangent cones in limit spaces with lower Ricci curvature bounds.  Journal f\"ur die reine und angewandte Mathematik (Crelles Journal) (2015): n. pag.

\bibitem{KiLa} Kitabeppu Y, Lakzian S. Characterization of low dimensional $\op{RCD}^*(K,N)$ spaces, Anal.Geom.Metr.Spaces 2016, 4: 187-215

\bibitem{LS} A. Lytchak, S. Stadler, Ricci curvature in dimension 2. J. Eur. Math. Soc. 25 (2023), no. 3, pp. 845–867

\bibitem{Na} A. Naber, Conjectures and open questions on the structure and regularity of spaces with lower Ricci curvature bounds, SIGMA Symmetry Integrability Geom. Methods Appl. 16 (2020), 8 pp.


\bibitem{Zhou} S. Zhou, Examples of Ricci limit spaces with infinite holes, preprint, arxiv:2404.00619

\end{thebibliography}
\end{document}